\documentclass[11pt]{article}

\usepackage{amsmath,amsthm}
\usepackage{amssymb}
\usepackage{hyperref}
\usepackage{color}

\allowdisplaybreaks

\begin{document}

\title{Explicit evaluation of some integrals involving polylogarithm functions}
\author{
Rusen Li
\\
\small School of Mathematics\\
\small Shandong University\\
\small Jinan 250100 China\\
\small \texttt{limanjiashe@163.com}
}

\date{
\small 2020 MR Subject Classifications: 33E20, 11B37, 11M06
}

\maketitle

\def\stf#1#2{\left[#1\atop#2\right]}
\def\sts#1#2{\left\{#1\atop#2\right\}}
\def\e{\mathfrak e}
\def\f{\mathfrak f}

\newtheorem{theorem}{Theorem}
\newtheorem{Prop}{Proposition}
\newtheorem{Cor}{Corollary}
\newtheorem{Lem}{Lemma}
\newtheorem{Example}{Example}
\newtheorem{Remark}{Remark}
\newtheorem{Definition}{Definition}
\newtheorem{Conjecture}{Conjecture}
\newtheorem{Problem}{Problem}

\begin{abstract}
In this paper, we give explicit evaluation for some integrals involving polylogarithm functions of types $\int_{0}^{x}t^{m} Li_{p}(t)\mathrm{d}t$ and $\int_{0}^{x}\log^{m}(t) Li_{p}(t)\mathrm{d}t$. Some more integrals involving the logarithm function will also be derived.
\end{abstract}

\section{Introduction and preliminaries}
Let $\mathbb Z$, $\mathbb N$, $\mathbb N_{0}$ and $\mathbb C$ denote the set of integers, positive integers, nonnegative integers and complex numbers, respectively. The well-known polylogarithm function is defined as
$$
Li_{p}(x):=\sum_{n=1}^\infty \frac{x^n}{n^p}\quad (\lvert x \lvert \leq 1,\quad p \in \mathbb N_{0})\,.
$$
Note that when $p=1$, $-Li_{1}(x)$ is the logarithm function $\log(1-x)$. Furthermore, $Li_{n}(1)=\zeta(n)$, where $\zeta(s)$ denotes the Riemann zeta function which is defined as $\zeta(s):=\sum_{n=1}^\infty n^{-s}$. The famous generalized harmonic numbers of order $m$ is defined by the partial sum of the Riemann Zeta function $\zeta(m)$ as:
$$
H_n^{(m)}:=\sum_{j=1}^n \frac{1}{j^m} \quad (n, m \in \mathbb N)\,.
$$

Before going further, we introduce some notations. Let $p \in \mathbb N$ and $m \in \mathbb N_{0}$, define
\begin{align*}
&J_{0}(m,p):=\int_{0}^{1}x^{m} Li_{p}(x)\mathrm{d}x\,.
\end{align*}
Let $p, m \in \mathbb N_{0}$ and $0 \leq x \leq 1$, define
\begin{align*}
&J_{0}(m,p,x):=\int_{0}^{x}t^{m} Li_{p}(t)\mathrm{d}t\,,
\quad J_{1}(m,p,x):=\int_{0}^{x}\log^{m}(t) Li_{p}(t)\mathrm{d}t\,.
\end{align*}
Let $p, q \in \mathbb N$ and $m \in \mathbb Z$ with $m \geq -2$, define
\begin{align*}
&J(m,p,q):=\int_{0}^{1}x^{m} Li_{p}(x)Li_{q}(x)\mathrm{d}x\,.
\end{align*}
Let $p, q \in \mathbb N_{0}$ with $p+q \geq 1$, $r \in \mathbb N$, define
\begin{align*}
&K(r,p,q):=\int_{0}^{1} \frac{\log^{r}(x) Li_{p}(x)Li_{q}(x)}{x}\mathrm{d}x\,.
\end{align*}
Freitas \cite{Freitas} showed that integrals $J_{0}(m,p)$, $J(m,p,q)$ and $K(r,p,q)$ satisfy the following recurrence relations:
\begin{align*}
&J_{0}(m,q)=\frac{\zeta(q)}{m+1}-\frac{1}{m+1}J_{0}(m,q-1)\quad (q \geq 2, m \geq 0)\,,\\
&J(m,p,q)=\frac{\zeta(p)\zeta(q)}{m+1}-\frac{1}{m+1}\bigg(J(m,p-1,q)+J(m,p,q-1)\bigg)\\
&(p, q\geq 2, m \in \mathbb N_{0}\cup\{-2\})\,,\\
&K(r,p,q)=-\frac{1}{r+1}\bigg(K(r+1,p-1,q)+K(r+1,p,q-1)\bigg)\quad (p, q, r \in \mathbb N)\,.
\end{align*}
From this Freitas proved that integrals $K(r,p,q)$ with $p+q+r$ even and $J(m,p,q)$ could be reduced to zeta values. Note that the proof was constructive, Freitas didn't give explicit evaluations for these integrals. On the contrary, Freitas \cite{Freitas} gave explicit evaluations for $J(-1,p,q)$ and $K(r,0,q)$ with $r+q$ even.

An anonymous reviewer told the author that Sofo \cite{Sofo2016,Sofo2018,Sofo2020}and Xu \cite{xu1,xu2,xu3} had made many progresses in the area of integrals involving polylogarithm functions, which the author was initially unaware of. It is known to all that polylogarithmic functions are intrinsically connected with sums of harmonic numbers. For instance, Sofo \cite{Sofo2016} developed closed form representations for infinite series containing generalized harmonic numbers of type
$$
\sum_{n=1}^\infty\frac{(-1)^{n+1}H_n^{(3)}}{n^{p}\binom{n+k}{k}}\quad (p=0, 1)\,.
$$
The author \cite{Lirusensen} gave closed form representations for generalized hyperharmonic number sums with reciprocal binomial coefficients, which greatly extend Sofo's result. Sofo \cite{Sofo2016} also obtained explicit evaluations for some integrals involving polylogarithm functions. Motivated by the work of Freitas \cite{Freitas}, Sofo \cite{Sofo2018} investigated the representations of integrals of polylogarithms with negative argument of the type
$$
\int_{0}^{1}x^{m} Li_{p}(-x)Li_{q}(-x)\mathrm{d}x\,
$$
for $m \ge -2$, and for integers $p$ and $q$. For $m=-2, -1, 0$, Sofo also gave explicit
representations of the integral in terms of Euler sums and for $m \ge 0$, Sofo obtained a
recurrence relation for the integral. As a more general consideration, Sofo \cite{Sofo2020} considered integrals of polylogarithms with alternating argument of the type
$$
\int_{0}^{1}x^{m} Li_{p}(x)Li_{q}(-x)\mathrm{d}x\,
$$
for integers $p$ and $q$. Similarly, for $m=-2, -1, 0$, Sofo gave explicit representations of the integral in terms of Euler sums. Some more integrals involving polylogarithms were obtained.
Xu \cite{xu1} showed that quadratic Euler sums of the form
\begin{align*}
\sum_{n=1}^\infty \frac{H_n H_n^{(m)}}{n^{p}}\quad(m+p\leq 8)\,,
\end{align*}
and some integrals of polylogarithm functions of the form
\begin{align*}
\int_{0}^{1} \frac{Li_{r}(x) Li_{p}(x) Li_{q}(x)}{x}\mathrm{d}x\quad(r+p+q\leq 8)\,
\end{align*}
can be written in terms of Riemann zeta values. It is interesting that integrals of polylogarithm functions can be related to multiple zeta (star) values. By using integrals of polylogarithm functions, Xu \cite{xu2} gave explicit expressions for some restricted multiple zeta (star) values. Some of lemmas used by Xu \cite{xu2} were also re-discovered by the author in different forms. Furthermore, by using the iterated integral representation of multiple polylogarithm functions, Xu \cite{xu3} proved some conjectures proposed by J. M. Borwein, D. M. Bradley and D. J. Broadhurst \cite{Borwein}. Xu also obtained numerous formulas for alternating multiple zeta values.

In this paper, we mainly give explicit expressions for integrals of types $J_{0}(m,p,x)$, $J_{1}(m,p,x)$, $J(m,p,q)$ and $K(r,p,q)$. In addition, some more explicit formulas for integrals involving the logarithm function of types
\begin{align*}
\int_{0}^{x}\frac{\log^{m}(1-t)}{t^{n}}\mathrm{d}t\,,\quad \int_{0}^{x}\frac{\log^{m}(1+t)}{t^{n}}\mathrm{d}t\,,\quad
\int_{0}^{x}\frac{\log^{m}(t)}{(1-t)^{n}}\mathrm{d}t\,\quad (m, n \in \mathbb N, m\geq n)
\end{align*}
will also be derived.

\section{Integrals involving logarithm function}

De Doelder \cite{Doelder} used the integral $\int_{0}^{x}\frac{\log^{2}(1-t)}{t}\mathrm{d}t$ to evaluate infinite series of type $\sum_{n=1}^\infty \frac{H_{n}}{n^{2}}x^{n}$. As a natural consideration, The author \cite{LIRUSEN} gave explicit evaluations for infinite series involving generalized (alternating) harmonic numbers of types
$\sum_{n=1}^\infty \frac{H_{n}}{n^{3}}x^{n}$,
$\sum_{n=1}^\infty \frac{H_{n}}{n^{3}}(-x)^{n}$,
$\sum_{n=1}^\infty \frac{H_{n}^{(2)}}{n}x^{n}$,
$\sum_{n=1}^\infty \frac{H_{n}^{(2)}}{n}(-x)^{n}$,
$\sum_{n=1}^\infty \frac{H_{n}^{(2)}}{n^{2}}x^{n}$,
$\sum_{n=1}^\infty \frac{H_{n}^{(2)}}{n^{2}}(-x)^{n}$,
$\sum_{n=1}^\infty \frac{\overline{H}_{n}}{n}x^{n}$,
$\sum_{n=1}^\infty \frac{\overline{H}_{n}^{(2)}}{n}x^{n}$,
$\sum_{n=1}^\infty \frac{\overline{H}_{n}^{(2)}}{n}(-x)^{n}$ in terms of polylogarithm functions. However, it seems difficult to give explicit expressions for infinite series of types
$\sum_{n=1}^\infty \frac{H_{n}}{n^{4}}x^{n}$ and
$\sum_{n=1}^\infty \frac{H_{n}^{(2)}}{n^{3}}x^{n}$, since the integrals $\int_{0}^{x}\frac{\log^{2}(t)\log^{2}(1-t)}{t}\mathrm{d}t$ and $\int_{0}^{x}\frac{\log^{2}(t)Li_{2}(t)}{1-t}\mathrm{d}t$ are not known to be related to the polylogarithm functions, even with the help of a mathematical package. It is interesting to evaluate similar type integrals, e.g., $\int_{0}^{x}\frac{\log^{m}(1-t)}{t^{n}}\mathrm{d}t$. Before going further, We introduce some notations.
\begin{Definition}\label{def1}
For $m, n \in \mathbb N$ with $m \geq n$ and $0 \leq x \leq 1$, define the quantities $A(m,n,x)$, $B(m,n,x)$ and $C(m,n,x)$ as
\begin{align*}
&A(m,n,x):=\int_{0}^{x}\frac{\log^{m}(1-t)}{t^{n}}\mathrm{d}t\,,\quad
B(m,n,x):=\int_{0}^{x}\frac{\log^{m}(1+t)}{t^{n}}\mathrm{d}t\,,\\
&C(m,n,x):=\int_{0}^{x}\frac{\log^{m}(t)}{(1-t)^{n}}\mathrm{d}t\,.
\end{align*}
\end{Definition}

\begin{Lem}\label{lem9}
Let $m \in \mathbb N$ and $0 \leq x \leq 1$, then we have
\begin{align*}
&\quad A(m,1,x)\\
&=\log(x)\log^{m}(1-x)+\sum_{k=0}^{m-2}(-1)^{k}(m-k)_{k+1}\log^{m-k-1}(1-x)Li_{k+2}(1-x)\\
&\quad +(-1)^{m-1}m!Li_{m+1}(1-x)+(-1)^{m}m!\zeta(m+1)\,.
\end{align*}
In particular, we have $A(m,1,1)=(-1)^{m}m!\zeta(m+1)$.
\end{Lem}
\begin{proof}
From the definition of $A(m,1,x)$, by using integration by parts, we can write
\begin{align*}
&\quad A(m,1,x)\\
&=\log(x)\log^{m}(1-x)-m \int_{1}^{1-x}\frac{\log(1-t)\log^{m-1}(t)}{t}\mathrm{d}t\\
&=\log(x)\log^{m}(1-x)+m \log^{m-1}(1-x)Li_{2}(1-x)\\
&\quad -m(m-1)\int_{1}^{1-x}\frac{Li_{2}(t)\log^{m-2}(t)}{t}\mathrm{d}t\\
&=\log(x)\log^{m}(1-x)+\sum_{k=0}^{m-2}(-1)^{k}(m-k)_{k+1}\log^{m-k-1}(1-x)Li_{k+2}(1-x)\\
&\quad +(-1)^{m-1}m!Li_{m+1}(1-x)+(-1)^{m}m!Li_{m+1}(1)\,.
\end{align*}
\end{proof}

\begin{Lem}[\cite{xu2}]\label{lem10}
Let $m \in \mathbb N$ and $x \geq 0$, then we have
\begin{align*}
B(m,1,x)
&=\log(x)\log^{m}(1+x)-\frac{m}{m+1}\log^{m+1}(1+x)+m!\zeta(m+1)\\
&\quad -\sum_{i=1}^{m}\binom{m}{i}i!\log^{m-i}(1+x)Li_{i+1}(\frac{1}{1+x})\,.
\end{align*}
In particular, we have
$$
B(m,1,1)
=-\frac{m}{m+1}\log^{m+1}(2)+m!\zeta(m+1)-\sum_{i=1}^{m}\binom{m}{i}i!\log^{m-i}(2)Li_{i+1}(\frac{1}{2})\,.
$$
\end{Lem}
\begin{proof}
From the definition of $B(m,1,x)$, by using integration by parts, we can write
\begin{align*}
&\quad B(m,1,x)\\
&=\log(x)\log^{m}(1+x)-m \int_{0}^{x}\frac{\log(t)\log^{m-1}(1+t)}{1+t}\mathrm{d}t\\
&=\log(x)\log^{m}(1+x)-m \int_{0}^{x}\log^{m-1}(1+t)\bigg(\frac{\mathrm{d}Li_{2}(\frac{1}{1+t})}{\mathrm{d}t}
+\frac{\log(1+t)}{1+t}\bigg)\mathrm{d}t\\
&=\log(x)\log^{m}(1+x)-\frac{m}{m+1}\log^{m+1}(1+x)-m\log^{m-1}(1+x)Li_{2}(\frac{1}{1+x})\\
&\quad +m(m-1) \int_{0}^{x}\frac{Li_{2}(\frac{1}{1+t})\log^{m-2}(1+t)}{1+t}\mathrm{d}t\\
&=\log(x)\log^{m}(1+x)-\frac{m}{m+1}\log^{m+1}(1+x)+m!Li_{m+1}(1)\\
&\quad -\sum_{i=1}^{m}\binom{m}{i}i!\log^{m-i}(1+x)Li_{i+1}(\frac{1}{1+x})\,.
\end{align*}
\end{proof}

\begin{Prop}\label{prop}
Let $m \in \mathbb N$ and $0 \leq x \leq 1$, then we have
\begin{align*}
&\quad C(m,1,x)\\
&=-\log(1-x)\log^{m}(x)+m\sum_{i=2}^{m+1}(-1)^{i-1}\binom{m-1}{i-2}(i-2)!\log^{m+1-i}(x)Li_{i}(x)\,.
\end{align*}
In particular, we have $C(m,1,1)=(-1)^{m}m!\zeta(m+1)$.
\end{Prop}
\begin{proof}
From the definition of $C(m,1,x)$, by using integration by parts, we can write
\begin{align*}
&\quad C(m,1,x)\\
&=-\log(1-x)\log^{m}(x)+m \int_{0}^{x}\frac{\log(1-t)\log^{m-1}(t)}{t}\mathrm{d}t\\
&=-\log(1-x)\log^{m}(x)-m \log^{m-1}(x)Li_{2}(x)+m(m-1)\int_{0}^{x}\frac{Li_{2}(t)\log^{m-2}(t)}{t}\mathrm{d}t\\
&=-\log(1-x)\log^{m}(x)+m\sum_{i=2}^{m+1}(-1)^{i-1}\binom{m-1}{i-2}(i-2)!\log^{m+1-i}(x)Li_{i}(x)\,.
\end{align*}
\end{proof}

Now we develop explicit expressions for $A(m,n,x)$, $B(m,n,x)$ and $C(m,n,x)$.
\begin{theorem}\label{maintheorem5}
Let $m, n \in \mathbb N$ with $m \geq n \geq 2$ and $0 < x \leq 1$, then we have
\begin{align*}
&\quad A(m,n,x)\\
&=\sum_{y=0}^{n-2}\binom{m}{y}y!(-1)^{y+1}\sum_{i_{0}=n}^{n}\frac{1}{i_{0}-1}
\sum_{i_{1}=2}^{i_{0}-1}\frac{1}{i_{1}-1}\cdots\sum_{i_{y}=2}^{i_{y-1}-1}\frac{1}{i_{y}-1}
\cdot\frac{\log^{m-y}(1-x)}{x^{i_{y}-1}}\\
&\quad +\sum_{y=0}^{n-2}\binom{m}{y}y!(-1)^{y}\sum_{i_{0}=n}^{n}\frac{1}{i_{0}-1}
\sum_{i_{1}=2}^{i_{0}-1}\frac{1}{i_{1}-1}\cdots\sum_{i_{y}=2}^{i_{y-1}-1}\frac{1}{i_{y}-1}
\log^{m-y}(1-x)\\
&\quad +\sum_{y=0}^{n-2}\binom{m}{y+1}(y+1)!(-1)^{y+1}\sum_{i_{0}=n}^{n}\frac{1}{i_{0}-1}
\sum_{i_{1}=2}^{i_{0}-1}\frac{1}{i_{1}-1}\cdots\sum_{i_{y}=2}^{i_{y-1}-1}\frac{1}{i_{y}-1}\\
&\qquad \times A(m-y-1,1,x)\,,
\end{align*}
and
\begin{align*}
&\quad C(m,n,x)\\
&=\frac{(-1)^{m}m!}{n-1}\sum_{y=0}^{n-2}\zeta(m-y)
\sum_{i_{1}=2}^{i_{0}-1}\frac{1}{i_{1}-1}\cdots\sum_{i_{y}=2}^{i_{y-1}-1}\frac{1}{i_{y}-1}
+\sum_{y=0}^{n-2}\binom{m}{y}y!(-1)^{y}\\
&\quad \times \sum_{i_{0}=n}^{n}\frac{1}{i_{0}-1}
\sum_{i_{1}=2}^{i_{0}-1}\frac{1}{i_{1}-1}\cdots\sum_{i_{y}=2}^{i_{y-1}-1}\frac{1}{i_{y}-1}
\bigg(\frac{\log^{m-y}(x)}{(1-x)^{i_{y}-1}}-\log^{m-y}(x)\bigg)\\
&\quad +\sum_{y=0}^{n-2}\binom{m}{y+1}(y+1)!(-1)^{y}\sum_{i_{0}=n}^{n}\frac{1}{i_{0}-1}
\sum_{i_{1}=2}^{i_{0}-1}\frac{1}{i_{1}-1}\cdots\sum_{i_{y}=2}^{i_{y-1}-1}\frac{1}{i_{y}-1}\\
&\qquad \times A(m-y-1,1,1-x)\,,
\end{align*}
where $A(m-y-1,1,x)$ and $A(m-y-1,1,1-x)$ are given in Lemma \ref{lem9}. In particular, we have
\begin{align*}
A(m,n,1)=C(m,n,1)=\frac{(-1)^{m}m!}{n-1}\sum_{y=0}^{n-2}\zeta(m-y)
\sum_{i_{1}=2}^{i_{0}-1}\frac{1}{i_{1}-1}\cdots\sum_{i_{y}=2}^{i_{y-1}-1}\frac{1}{i_{y}-1}\,.
\end{align*}
\end{theorem}
\begin{proof}
From the definition of $A(m,n,x)$, when $n \geq 2$, by using integration by parts, we can write
\begin{align*}
&\quad A(m,n,x)\\
&=-\frac{1}{n-1}\cdot\frac{\log^{m}(1-x)}{x^{n-1}}
-\frac{m}{n-1}\int_{0}^{x}\frac{\log^{m-1}(1-t)}{t^{n-1}(1-t)}\mathrm{d}t\\
&=-\frac{1}{n-1}\cdot\frac{\log^{m}(1-x)}{x^{n-1}}-\frac{m}{n-1}\sum_{i=1}^{n-1}\int_{0}^{x}\frac{\log^{m-1}(1-t)}{t^{i}}\mathrm{d}t\\
&\quad -\frac{m}{n-1}\int_{0}^{x}\frac{\log^{m-1}(1-t)}{1-t}\mathrm{d}t\\
&=-\frac{1}{n-1}\cdot\frac{\log^{m}(1-x)}{x^{n-1}}+\frac{1}{n-1}\log^{m}(1-x)
-\frac{m}{n-1}\sum_{i=1}^{n-1}A(m-1,i,x)\,,
\end{align*}
successive application of the above relation $n-2$ times, we can obtain that
\begin{align*}
&\quad A(m,n,x)\\
&=-\frac{1}{n-1}\cdot\frac{\log^{m}(1-x)}{x^{n-1}}+\frac{1}{n-1}\log^{m}(1-x)-\frac{m}{n-1}A(m-1,1,x)\\
&\quad +\frac{m}{n-1}\sum_{i_{1}=2}^{n-1}\frac{1}{i_{1}-1}\cdot\frac{\log^{m-1}(1-x)}{x^{i_{1}-1}}
-\frac{m}{n-1}\sum_{i_{1}=2}^{n-1}\frac{1}{i_{1}-1}\log^{m-1}(1-x)\\
&\quad +\frac{m}{n-1}\sum_{i_{1}=2}^{n-1}\frac{m}{i_{1}-1}A(m-2,1,x)
+\frac{m}{n-1}\sum_{i_{1}=2}^{n-1}\frac{m}{i_{1}-1}\sum_{i_{2}=2}^{i_{1}-1}A(m-2,i_{2},x)\\
&=\sum_{y=0}^{n-2}\binom{m}{y}y!(-1)^{y+1}\sum_{i_{0}=n}^{n}\frac{1}{i_{0}-1}
\sum_{i_{1}=2}^{i_{0}-1}\frac{1}{i_{1}-1}\cdots\sum_{i_{y}=2}^{i_{y-1}-1}\frac{1}{i_{y}-1}
\cdot\frac{\log^{m-y}(1-x)}{x^{i_{y}-1}}\\
&\quad +\sum_{y=0}^{n-2}\binom{m}{y}y!(-1)^{y}\sum_{i_{0}=n}^{n}\frac{1}{i_{0}-1}
\sum_{i_{1}=2}^{i_{0}-1}\frac{1}{i_{1}-1}\cdots\sum_{i_{y}=2}^{i_{y-1}-1}\frac{1}{i_{y}-1}
\log^{m-y}(1-x)\\
&\quad +\sum_{y=0}^{n-2}\binom{m}{y+1}(y+1)!(-1)^{y+1}\sum_{i_{0}=n}^{n}\frac{1}{i_{0}-1}
\sum_{i_{1}=2}^{i_{0}-1}\frac{1}{i_{1}-1}\cdots\sum_{i_{y}=2}^{i_{y-1}-1}\frac{1}{i_{y}-1}\\
&\qquad \times A(m-y-1,1,x)\,.
\end{align*}
Note that $C(m,n,x)=A(m,n,1)-A(m,n,1-x)$, thus we get the desired result.
\end{proof}

\begin{theorem}\label{maintheorem6}
Let $m, n \in \mathbb N$ with $m \geq n \geq 2$ and $x > 0$, then we have
\begin{align*}
&\quad B(m,n,x)\\
&=\sum_{y=0}^{n-2}\binom{m}{y}y!\sum_{i_{0}=n}^{n}\frac{1}{i_{0}-1}
\sum_{i_{1}=2}^{i_{0}-1}\frac{1}{i_{1}-1}\cdots\sum_{i_{y}=2}^{i_{y-1}-1}\frac{1}{i_{y}-1}
\cdot\frac{(-1)^{n+i_{y}+y+1}\log^{m-y}(1+x)}{x^{i_{y}-1}}\\
&\quad +\sum_{y=0}^{n-2}\binom{m}{y}y!\sum_{i_{0}=n}^{n}\frac{1}{i_{0}-1}
\sum_{i_{1}=2}^{i_{0}-1}\frac{1}{i_{1}-1}\cdots\sum_{i_{y}=2}^{i_{y-1}-1}\frac{1}{i_{y}-1}
\cdot(-1)^{n+y+1}\log^{m-y}(1+x)\\
&\quad +\sum_{y=0}^{n-2}\binom{m}{y+1}(y+1)!\sum_{i_{0}=n}^{n}\frac{1}{i_{0}-1}
\sum_{i_{1}=2}^{i_{0}-1}\frac{1}{i_{1}-1}\cdots\sum_{i_{y}=2}^{i_{y-1}-1}\frac{1}{i_{y}-1}\\
&\qquad \times (-1)^{n+y}B(m-y-1,1,x)\,,
\end{align*}
where $B(m-y-1,1,x)$ are given in Lemma \ref{lem10}. In particular, we have
\begin{align*}
B(m,n,1)
&=\sum_{y=0}^{n-2}\binom{m}{y}y!\sum_{i_{0}=n}^{n}\frac{1}{i_{0}-1}
\sum_{i_{1}=2}^{i_{0}-1}\frac{1}{i_{1}-1}\cdots\sum_{i_{y}=2}^{i_{y-1}-1}\frac{1}{i_{y}-1}\\
&\qquad \times\log^{m-y}(2)(-1)^{n+y+1}((-1)^{i_{y}}+1)\\
&\quad +\sum_{y=0}^{n-2}\binom{m}{y+1}(y+1)!\sum_{i_{0}=n}^{n}\frac{1}{i_{0}-1}
\sum_{i_{1}=2}^{i_{0}-1}\frac{1}{i_{1}-1}\cdots\sum_{i_{y}=2}^{i_{y-1}-1}\frac{1}{i_{y}-1}\\
&\qquad \times (-1)^{n+y}\bigg\{-\frac{m-y-1}{m-y}\log^{m-y}(2)+(m-y-1)!\zeta(m-y)\\
&\qquad -\sum_{i=1}^{m-y-1}\binom{m-y-1}{i}i!\log^{m-y-1-i}(2)Li_{i+1}(\frac{1}{2})\bigg\}\,.
\end{align*}
\end{theorem}
\begin{proof}
From the definition of $B(m,n,x)$, when $n \geq 2$, by using integration by parts, we can write
\begin{align*}
B(m,n,x)
&=-\frac{1}{n-1}\cdot\frac{\log^{m}(1+x)}{x^{n-1}}
+\frac{m}{n-1}\int_{0}^{x}\frac{\log^{m-1}(1+t)}{t^{n-1}(1+t)}\mathrm{d}t\\
&=-\frac{1}{n-1}\cdot\frac{\log^{m}(1+x)}{x^{n-1}}
+\frac{m}{n-1}\sum_{i=1}^{n-1}(-1)^{n-1-i}\int_{0}^{x}\frac{\log^{m-1}(1+t)}{t^{i}}\mathrm{d}t\\
&\quad +\frac{m}{n-1}\int_{0}^{x}\frac{(-1)^{n-1}\log^{m-1}(1+t)}{1+t}\mathrm{d}t\\
&=-\frac{1}{n-1}\cdot\frac{\log^{m}(1+x)}{x^{n-1}}+\frac{(-1)^{n-1}}{n-1}\log^{m}(1+x)\\
&\quad +\frac{m}{n-1}\sum_{i=1}^{n-1}(-1)^{n-1-i}B(m-1,i,x)\,,
\end{align*}
similar with Theorem \ref{maintheorem5}, successive application of the above relation $n-2$ times, we get the desired result.
\end{proof}

\section{Integrals involving polylogarithms}

In this section, we develop explicit expressions for integrals of types $J_{0}(m,p,x)$, $J_{1}(m,p,x)$, $J(m,p,q)$ and $K(r,p,q)$. Before going further, we introduce some notations and lemmata. Following Flajolet-Salvy's paper \cite{Flajolet}, we write the classical linear Euler sums as
$$
S_{p,q}^{+,+}:=\sum_{n=1}^\infty \frac{H_n^{(p)}}{{n}^{q}}\,.
$$

\begin{Lem}[\cite{xu2,Lirusensen}]\label{lem1}
Let $n, m \in \mathbb N_{0}$ and $x \geq 0$, defining
\begin{align*}
L(n, m, x):=\int_{0}^{x}y^{n}\log^{m}(y) \mathrm{d}y\,,
\end{align*}
then we have
\begin{align*}
L(n, m, x)=\frac{x^{n+1}}{n+1}\sum_{j=0}^{m}\frac{(m+1-j)_{j}}{(n+1)^j}(-1)^{j}\log^{m-j}(x)\,,
\end{align*}
where $(t)_{n}=t(t+1)\cdots(t+n-1)$ is the Pochhammer symbol. In particular, we have $L(n, m, 1)=\frac{m!(-1)^m}{(n+1)^{m+1}}$.
\end{Lem}

\begin{Lem}[\cite{Lirusensen}]\label{lem2}
Let $n, m \in \mathbb N_{0}$ and $0 \leq x \leq 1$, defining
\begin{align*}
M(n, m, x):=\int_{x}^{1}y^{n}\log^{m} (1-y) \mathrm{d}y\,,
\end{align*}
then we have
\begin{align*}
M(n, m, x)=\sum_{j=0}^{n}\binom{n}{j}(-1)^{j}\frac{(1-x)^{j+1}}{j+1}
\sum_{i=0}^{m}\frac{(m+1-i)_{i}}{(j+1)^i}(-1)^{i}\log^{m-i} (1-x)\,.
\end{align*}
In particular, we have $M(n, m, 0)=(-1)^{m}m!\sum_{j=0}^{n}\binom{n}{j}\frac{(-1)^{j}}{(j+1)^{m+1}}$.
\end{Lem}

\begin{Lem}\label{lem3}
Let $n, m \in \mathbb N_{0}$ and $0 \leq x \leq 1$, then we can obtain that
\begin{align*}
&\int_{0}^{x}y^{n}\log^{m} (1-y) \mathrm{d}y\\
&=(-1)^{m}m!\sum_{j=0}^{n}\binom{n}{j}\frac{(-1)^{j}}{(j+1)^{m+1}}\\
&\quad-\sum_{j=0}^{n}\binom{n}{j}(-1)^{j}\frac{(1-x)^{j+1}}{j+1}
\sum_{i=0}^{m}\frac{(m+1-i)_{i}}{(j+1)^i}(-1)^{i}\log^{m-i} (1-x)\,.
\end{align*}
\begin{proof}
Note that
$$
\int_{0}^{x}y^{n}\log^{m} (1-y) \mathrm{d}y=M(n, m, 0)-M(n, m, x)\,,
$$
with the help of Lemma \ref{lem2}, we get the desired result.
\end{proof}
\end{Lem}

\begin{theorem}\label{maintheorem1}
Let $p \in \mathbb N$ and $m \in \mathbb N_{0}$, then we have
\begin{align*}
&J_{0}(m,p,x)\\
&=\sum_{j=2}^{p}\frac{(-1)^{p-j}}{(m+1)^{p+1-j}}x^{m+1}Li_{j}(x)
+\frac{(-1)^{p-1}}{(m+1)^{p-1}}\bigg(\sum_{j=0}^{m}\binom{m}{j}\frac{(-1)^{j}}{(j+1)^{2}}\\
&\quad+\sum_{j=0}^{m}\binom{m}{j}(-1)^{j}\frac{(1-x)^{j+1}}{j+1}
\sum_{i=0}^{1}\frac{(2-i)_{i}}{(j+1)^i}(-1)^{i}\log^{1-i} (1-x)\bigg)\,.
\end{align*}
In particular, $J_{0}(m,p)$ can be reduced to zeta values and harmonic numbers:
$$
J_{0}(m,p)=J_{0}(m,p,1)
=\sum_{j=2}^{p}\frac{(-1)^{p-j}}{(m+1)^{p+1-j}}\zeta(j)+\frac{(-1)^{p-1}}{(m+1)^{p}}H_{m+1}\,.
$$
Note that we have used the fact \cite{Lirusensen}
$$
H_{m+1}=(m+1)\sum_{j=0}^{m}\binom{m}{j}\frac{(-1)^{j}}{(j+1)^{2}}\,.
$$
\end{theorem}
\begin{proof}
By using integration by parts, we have
\begin{align*}
&\quad J_{0}(m,p,x)\\
&=\frac{1}{m+1}x^{m+1} Li_{p}(x)-\frac{1}{m+1}\int_{0}^{x}t^{m} Li_{p-1}(t)\mathrm{d}t\\
&=\frac{1}{m+1}x^{m+1} Li_{p}(x)-\frac{1}{m+1}J_{0}(m,p-1,x)\\
&=\frac{1}{m+1}x^{m+1} Li_{p}(x)-\frac{1}{(m+1)^{2}}x^{m+1} Li_{p-1}(x)+\frac{1}{(m+1)^{2}}J_{0}(m,p-2,x)\\
&=\sum_{j=2}^{p}\frac{(-1)^{p-j}}{(m+1)^{p+1-j}}x^{m+1}Li_{j}(x)+\frac{(-1)^{p-1}}{(m+1)^{p-1}}J_{0}(m,1,x)\,.
\end{align*}
Note that
$$
J_{0}(m,1,x)=-\int_{0}^{x}t^{m} \log(1-t)\mathrm{d}t\,,
$$
with the help of Lemma \ref{lem3}, we get the desired result.
\end{proof}

\begin{Lem}\label{lem4}
Let $m \in \mathbb N_{0}$, then we have
\begin{align*}
J_{1}(m,0,x)
&=x\sum_{j=0}^{m}(m+1-j)_{j}(-1)^{j+1}\log^{m-j}(x)-\log(1-x)\log^{m}(x)\\
&\quad+m\sum_{j=2}^{m+1}(-1)^{j-1}\binom{m-1}{j-2}(j-2)!Li_{j}(x)\log^{m+1-j}(x)\,.
\end{align*}
\begin{proof}
Note that
\begin{align*}
J_{1}(m,0,x)
&=-\int_{0}^{x}\log^{m}(t)\mathrm{d}t+\int_{0}^{x}\frac{\log^{m}(t)}{1-t}\mathrm{d}t\,.
\end{align*}
with the help of Lemma \ref{lem1} and Proposition \ref{prop}, we get the desired result.
\end{proof}
\end{Lem}

\begin{theorem}\label{maintheorem2}
Let $p \in \mathbb N$ and $m \in \mathbb N_{0}$, then we have
\begin{align*}
&J_{1}(m,p,x)\\
&=\sum_{y=1}^{p}\sum_{i_{1}=0}^{m-1}\cdots\sum_{i_{y}=0}^{m-i_{1}-\cdots-i_{y-1}-1}
m(m-1)\cdots(m-i_{1}-\cdots-i_{y}+1)\\
&\qquad \times(-1)^{i_{1}+\cdots+i_{y}+y-1} x Li_{p-y+1}(x)\log^{m-i_{1}-\cdots-i_{y}}(x)\\
&\quad +\sum_{y=1}^{p}\sum_{i_{1}=0}^{m-1}\cdots\sum_{i_{y-1}=0}^{m-i_{1}-\cdots-i_{y-2}-1}
m!(-1)^{i_{1}+\cdots+i_{y-1}+y-1}J_{1}(0,p-y+1,x)\\
&\quad +\sum_{i_{1}=0}^{m-1}\cdots\sum_{i_{p}=0}^{m-i_{1}-\cdots-i_{p-1}-1}
m(m-1)\cdots(m-i_{1}-\cdots-i_{p}+1)\\
&\qquad \times(-1)^{i_{1}+\cdots+i_{p}+p}J_{1}(m-i_{1}-\cdots-i_{p},0,x)\,,
\end{align*}
where $J_{1}(0,p-y+1,x)=J_{0}(0,p-y+1,x)$ are given Theorem \ref{maintheorem1} and $J_{1}(m-i_{1}-\cdots-i_{p},0,x)$ are given in Lemma \ref{lem4}.
\end{theorem}
\begin{proof}
By using integration by parts, we can obtain that
\begin{align*}
&\quad J_{1}(m,p,x)\\
&=x Li_{p}(x)\log^{m}(x)-\int_{0}^{x}\log^{m}(t) Li_{p-1}(t)\mathrm{d}t-m \int_{0}^{x}\log^{m-1}(t) Li_{p}(t)\mathrm{d}t\\
&=x Li_{p}(x)\log^{m}(x)-J_{1}(m,p-1,x)-m J_{1}(m-1,p,x)\\
&=x Li_{p}(x)\log^{m}(x)-m x Li_{p}(x)\log^{m-1}(x)-J_{1}(m,p-1,x)\\
&\quad +m J_{1}(m-1,p-1,x)+m(m-1) J_{1}(m-2,p,x)\\
&=\sum_{i=0}^{m-1}\binom{m}{i}i!(-1)^{i} x Li_{p}(x)\log^{m-i}(x)+(-1)^{m}m!J_{1}(0,p,x)\\
&\quad +\sum_{i=0}^{m-1}\binom{m}{i}i!(-1)^{i+1} J_{1}(m-i,p-1,x)\\
&=\sum_{i_{1}=0}^{m-1}\binom{m}{i_{1}}i_{1}!(-1)^{i_{1}} x Li_{p}(x)\log^{m-i_{1}}(x)+(-1)^{m}m!J_{1}(0,p,x)\\
&\quad +\sum_{i_{1}=0}^{m-1}\binom{m}{i_{1}}i_{1}!(-1)^{i_{1}+1}\bigg\{(-1)^{m-i_{1}}(m-i_{1})!J_{1}(0,p-1,x)\\
&\quad +\sum_{i_{2}=0}^{m-i_{1}-1}
\binom{m-i_{1}}{i_{2}}i_{2}!(-1)^{i_{2}} x Li_{p-1}(x)\log^{m-i_{1}-i_{2}}(x)\\
&\quad +\sum_{i_{2}=0}^{m-i_{1}-1}
\binom{m-i_{1}}{i_{2}}i_{2}!(-1)^{i_{2}+1} J_{1}(m-i_{1}-i_{2},p-2,x)\bigg\}\\
&=\sum_{y=1}^{p}\sum_{i_{1}=0}^{m-1}\cdots\sum_{i_{y}=0}^{m-i_{1}-\cdots-i_{y-1}-1}
m(m-1)\cdots(m-i_{1}-\cdots-i_{y}+1)\\
&\qquad \times(-1)^{i_{1}+\cdots+i_{y}+y-1} x Li_{p-y+1}(x)\log^{m-i_{1}-\cdots-i_{y}}(x)\\
&\quad +\sum_{y=1}^{p}\sum_{i_{1}=0}^{m-1}\cdots\sum_{i_{y-1}=0}^{m-i_{1}-\cdots-i_{y-2}-1}
m!(-1)^{i_{1}+\cdots+i_{y-1}+y-1}J_{1}(0,p-y+1,x)\\
&\quad +\sum_{i_{1}=0}^{m-1}\cdots\sum_{i_{p}=0}^{m-i_{1}-\cdots-i_{p-1}-1}
m(m-1)\cdots(m-i_{1}-\cdots-i_{p}+1)\\
&\qquad \times(-1)^{i_{1}+\cdots+i_{p}+p}J_{1}(m-i_{1}-\cdots-i_{p},0,x)\,.
\end{align*}
\end{proof}

\begin{Lem}[{{Abel's lemma on summation by parts} \cite{Abel,Chu}}]\label{lem5}
Let $\{f_k\}$ and $\{g_k\}$ be two sequences, and define the forward difference and backward difference, respectively, as
$$
\Delta\tau_k=\tau_{k+1}-\tau_k\quad\hbox{and}\quad \nabla\tau_k=\tau_k-\tau_{k-1}\,,
$$
then, there holds the relation:
\begin{align*}
\sum_{k=1}^\infty f_k\nabla g_k=\lim_{n\to\infty}f_n g_n-f_1 g_0-\sum_{k=1}^\infty g_k \Delta f_k \,.
\end{align*}
\end{Lem}

We now provide a criterion concerning the exchange of summation and integral for improper integrals.
\begin{Lem}\label{exchange}
Given a series of functions $\sum_{n=1}^\infty u_{n}(x), a\leq x \leq b$ with $u_{n}(x)\geq 0$, $\sum_{n=1}^\infty u_{n}(b)=\infty$ and $\sum_{n=1}^\infty u_{n}(x)$ converges for $a \leq x < b$. Suppose $u_{n}(x)$ is integrable (Riemann integrable or integrable as an improper integral) on $[a, b]$, the improper integral $\int_{a}^{b}\sum_{n=1}^\infty u_{n}(x)\mathrm{d}x$ converges, and $\sum_{n=1}^\infty u_{n}(x)$ internally closed uniform converges on $[a, b)$, i.e. for any $a\leq c < d <b$, $\sum_{n=1}^\infty u_{n}(x)$ converges uniformly on $[c, d]$, then we can exchange summation and integral, i.e.
$$
\int_{a}^{b}\sum_{n=1}^\infty u_{n}(x)\mathrm{d}x=\sum_{n=1}^\infty \int_{a}^{b} u_{n}(x)\mathrm{d}x\,.
$$
\end{Lem}
\begin{proof}
Since the improper integral $\int_{a}^{b}\sum_{n=1}^\infty u_{n}(x)\mathrm{d}x$ converges, then for any $\epsilon >0$, there exists $\delta >0$, s.t. $\lvert \int_{b-\delta}^{b}\sum_{n=1}^\infty u_{n}(x)\mathrm{d}x \rvert\leq\frac{\epsilon}{3}$. It is not hard to see that we can choose $\delta >0$ such that $b-\delta-a>0$. Note that, $\sum_{n=1}^\infty u_{n}(x)$ converges uniformly on $[a, b-\delta]$, then for fixed $\epsilon, \delta >0$, there exists $N > 0$, for any $M \geq N$, we have
$$
\left\lvert \sum_{n=M+1}^\infty u_{n}(x) \right\rvert\leq\frac{\epsilon}{3(b-\delta-a)}\quad x\in[a, b-\delta]\,.
$$
Thus we can obtain that
\begin{align*}
&\quad \left\lvert \sum_{n=1}^{M} \int_{a}^{b}u_{n}(x)\mathrm{d}x-\int_{a}^{b} \sum_{n=1}^\infty u_{n}(x)\mathrm{d}x \right\rvert\\
&\leq\left\lvert \int_{a}^{b-\delta}\sum_{n=1}^{M} u_{n}(x)-\sum_{n=1}^{\infty} u_{n}(x)\mathrm{d}x\right\rvert+\left\lvert\int_{b-\delta}^{b} \sum_{n=1}^{M} u_{n}(x)\mathrm{d}x \right\rvert+\left\lvert\int_{b-\delta}^{b} \sum_{n=1}^{\infty} u_{n}(x)\mathrm{d}x \right\rvert\\
&\leq(b-\delta-a)\frac{\epsilon}{3(b-\delta-a)}+\frac{\epsilon}{3}+\frac{\epsilon}{3}\\
&=\epsilon\,,
\end{align*}
which completes the proof.
\end{proof}

\begin{Lem}\label{lem6}
Let $m \in \mathbb N_{0}$ and $p \in \mathbb N$, then we have
\begin{align*}
&\quad J(m,p,1)\\
&=\sum_{j=2}^{p}(-1)^{p-j}\zeta(j)\bigg(\sum_{i=2}^{p+1-j}\frac{-1}{(m+1)^{p+2-j-i}}\zeta(i)
+\sum_{i=1}^{p+1-j}\frac{1}{(m+1)^{p+2-j-i}}H_{m+1}^{(i)}\bigg)\\
&\quad +(-1)^{p-1}\bigg\{\sum_{i=2}^{p}\frac{-1}{(m+1)^{p-i+1}}\bigg((1+\frac{i}{2})\zeta(i+1)
-\frac{1}{2}\sum_{k=1}^{i-2}\zeta(k+1)\zeta(i-k)\\
&\qquad -\sum_{n=1}^{m+1}\frac{H_{n}}{n^{i}}\bigg)+\frac{1}{(m+1)^{p}}
\bigg(H_{m+1}^{2}+\sum_{b=0}^{m}\frac{H_{m+1}-H_{b}}{m+1-b}\bigg)\bigg\}\,.
\end{align*}
\end{Lem}
\begin{proof}
From the definition of $J(m,p,1)$, we can write
\begin{align*}
J(m,p,1)
&=\sum_{n=1}^\infty \frac{1}{n} \int_{0}^{1}x^{m+n} Li_{p}(x)\mathrm{d}x\\
&=\sum_{j=2}^{p}(-1)^{p-j}\zeta(j)\sum_{n=1}^\infty \frac{1}{n(m+n+1)^{p+1-j}}+\sum_{n=1}^\infty \frac{(-1)^{p-1}H_{m+n+1}}{n(m+n+1)^{p}}\,.
\end{align*}
For the first part, by using fraction expansion, we have
\begin{align*}
&\quad \sum_{n=1}^\infty \frac{1}{n(m+n+1)^{p+1-j}}\\
&=\sum_{n=1}^\infty \bigg(\sum_{i=2}^{p+1-j}\frac{-1}{(m+1)^{p+2-j-i}}\cdot\frac{1}{(n+m+1)^{i}}
+\frac{1}{(m+1)^{p-j}}\cdot\frac{1}{n(n+m+1)}\bigg)\\
&=\sum_{i=2}^{p+1-j}\frac{-1}{(m+1)^{p+2-j-i}}\bigg(\zeta(i)-H_{m+1}^{(i)}\bigg)
+\frac{1}{(m+1)^{p-j+1}}H_{m+1}\\
&=\sum_{i=2}^{p+1-j}\frac{-1}{(m+1)^{p+2-j-i}}\zeta(i)
+\sum_{i=1}^{p+1-j}\frac{1}{(m+1)^{p+2-j-i}}H_{m+1}^{(i)}\,.
\end{align*}
For the second part, by using fraction expansion, we have
\begin{align*}
&\quad \sum_{n=1}^\infty \frac{H_{m+n+1}}{n(m+n+1)^{p}}\\
&=\sum_{n=1}^\infty H_{m+n+1}\bigg(\sum_{i=2}^{p}\frac{-1}{(m+1)^{p-i+1}}\cdot\frac{1}{(n+m+1)^{i}}
+\frac{1}{(m+1)^{p-1}}\cdot\frac{1}{n(n+m+1)}\bigg)\\
&=\sum_{i=2}^{p}\frac{-1}{(m+1)^{p-i+1}}\bigg(\sum_{n=1}^\infty \frac{H_{n}}{n^{i}}-\sum_{n=1}^{m+1} \frac{H_{n}}{n^{i}}\bigg)+\frac{1}{(m+1)^{p-1}}\sum_{n=1}^\infty\frac{H_{m+n+1}}{n(n+m+1)}\,.
\end{align*}
Note that
\begin{align*}
S_{1,i}^{+,+}=\sum_{n=1}^\infty \frac{H_{n}}{n^{i}}=(1+\frac{i}{2})\zeta(i+1)
-\frac{1}{2}\sum_{k=1}^{i-2}\zeta(k+1)\zeta(i-k)\,.\quad \cite{Flajolet}
\end{align*}
Set
$$
f_{n}:=H_{n+m+1}\quad\hbox{and}\quad
g_{n}:=\frac{1}{n+1}+\cdots+\frac{1}{n+m+1}\,,
$$
by using Lemma \ref{lem5}, we have
\begin{align*}
&\quad -\sum_{n=1}^\infty \frac{(m+1) H_{n+m+1}}{n(n+m+1)}\\
&=\sum_{n=1}^\infty H_{n+m+1}\left(\bigg(\frac{1}{n+1}+\cdots+\frac{1}{n+m+1}\bigg)-\bigg(\frac{1}{n}+\cdots+\frac{1}{n+m}\bigg)\right)\\
&=-H_{m+2}\bigg(\frac{1}{1}+\cdots+\frac{1}{1+m}\bigg)-\sum_{n=1}^\infty \bigg(\frac{1}{n+1}+\cdots+\frac{1}{n+m+1}\bigg)\frac{1}{n+m+2}\\
&=-H_{m+1}^{2}-\sum_{n=1}^\infty \sum_{b=0}^{m}\frac{1}{n+b}\cdot\frac{1}{n+m+1}\\
&=-H_{m+1}^{2}-\sum_{b=0}^{m}\frac{1}{m+1-b}\bigg(H_{m+1}-H_{b}\bigg)\,.
\end{align*}
Combining the above results, we get the desired result.
\end{proof}

\begin{Remark}
In the proof of the above theorem, we exchange the order of summation and integration, i.e.
\begin{align*}
\int_{0}^{1}x^{m}\sum_{n=1}^\infty \frac{x^{n}}{n} Li_{p}(x)\mathrm{d}x=\sum_{n=1}^\infty \frac{1}{n} \int_{0}^{1}x^{m+n} Li_{p}(x)\mathrm{d}x\,.
\end{align*}
To verify this, we only need to note that for any $0<\delta<1$, $\frac{x^{m+n}}{n} Li_{p}(x)$ is monotonic increasing on the interval $[0, 1-\delta]$ and the series $\sum_{n=1}^\infty \frac{x^{m+n}}{n} Li_{p}(x)$ converges when $x=1-\delta$. With the help of Lemma \ref{exchange}, we get the desired result. The other cases can be checked in a similar manner.
\end{Remark}

Now we provide another formula for $J(m,p,1)$.
\begin{Lem}\label{lem7}
Let $m \in \mathbb N_{0}$ and $p \in \mathbb N$, then we have
\begin{align*}
J(m,p,1)
&=\sum_{i=2}^{p}\frac{(-1)^{p-i}}{(m+1)^{p-i+1}}\bigg((1+\frac{i}{2})\zeta(i+1)
-\frac{1}{2}\sum_{k=1}^{i-2}\zeta(k+1)\zeta(i-k)\\
&\quad +\sum_{k=2}^{i}(-1)^{i-k}\zeta(k)H_{m+1}^{(i-k+1)}
+\sum_{j=1}^{m+1}\frac{(-1)^{i-1}}{j^{i}}H_{j}\bigg)\\
&\quad +\frac{(-1)^{p-1}}{(m+1)^{p}}
\bigg(H_{m+1}^{2}+\sum_{b=0}^{m}\frac{H_{m+1}-H_{b}}{m+1-b}\bigg)\,.
\end{align*}
\end{Lem}
\begin{proof}
From the definition of $J(m,p,1)$, we can write
\begin{align*}
J(m,p,1)
&=-\sum_{n=1}^\infty \frac{1}{n^{p}} \int_{0}^{1}x^{m+n} \log(1-x)\mathrm{d}x\\
&=\sum_{n=1}^\infty \frac{1}{n^{p}} \frac{H_{m+n+1}}{m+n+1}\\
&=\sum_{n=1}^\infty H_{m+n+1}\bigg(\sum_{i=2}^{p}\frac{(-1)^{p-i}}{(m+1)^{p-i+1}}\cdot\frac{1}{n^{i}}
+\frac{(-1)^{p-1}}{(m+1)^{p-1}}\cdot\frac{1}{n(n+m+1)}\bigg)\,.
\end{align*}
For the first part, by using fraction expansion, we have
\begin{align*}
&\quad \sum_{n=1}^\infty \frac{H_{m+n+1}}{n^{i}}\\
&=\sum_{n=1}^\infty \frac{H_{n}}{n^{i}}
+\sum_{j=1}^{m+1}\sum_{n=1}^\infty \frac{1}{n^{i}(n+j)}\\
&=\sum_{n=1}^\infty \frac{H_{n}}{n^{i}}+\sum_{j=1}^{m+1}\sum_{n=1}^\infty \bigg(\sum_{k=2}^{i}\frac{(-1)^{i-k}}{j^{i-k+1}}\cdot\frac{1}{n^{k}}
+\frac{(-1)^{i-1}}{j^{i-1}}\cdot\frac{1}{n(n+j)}\bigg)\\
&=\sum_{n=1}^\infty \frac{H_{n}}{n^{i}}+\sum_{k=2}^{i}(-1)^{i-k}\zeta(k)H_{m+1}^{(i-k+1)}
+\sum_{j=1}^{m+1}\frac{(-1)^{i-1}}{j^{i}}H_{j}\,.
\end{align*}
For the second part, from the previous lemma we know that
\begin{align*}
&\quad \sum_{n=1}^\infty \frac{(m+1) H_{n+m+1}}{n(n+m+1)}
=H_{m+1}^{2}+\sum_{b=0}^{m}\frac{1}{m+1-b}\bigg(H_{m+1}-H_{b}\bigg)\,.
\end{align*}
Combining the above results, we get the desired result.
\end{proof}

When $p=1$, we have
$$
J(m,1,1)=\frac{1}{m+1}
\bigg(H_{m+1}^{2}+\sum_{b=0}^{m}\frac{H_{m+1}-H_{b}}{m+1-b}\bigg)\,.
$$
It is known that \cite{Devoto}
$$
J(m,1,1)=\frac{2}{m+1}\bigg(H_{m+1}^{(2)}+\sum_{k=1}^{m}\frac{H_{k}}{k+1}\bigg)\,,
$$
thus we have the following proposition:
\begin{Prop}
\begin{align*}
H_{m+1}^{(2)}
&=\frac{1}{2}\bigg(\sum_{j=0}^{m}\binom{m+1}{j+1}\frac{(-1)^{j}}{j+1}\bigg)^{2}
+\frac{1}{2}\sum_{b=0}^{m}\frac{1}{m+1-b}\bigg\{\sum_{j=0}^{m}\binom{m+1}{j+1}\frac{(-1)^{j}}{j+1}\\
&\quad -\sum_{j=0}^{b-1}\binom{b}{j+1}\frac{(-1)^{j}}{j+1}\bigg\}
-\sum_{k=1}^{m}\frac{1}{k+1}\sum_{j=0}^{k-1}\binom{k}{j+1}\frac{(-1)^{j}}{j+1}\,.
\end{align*}
\end{Prop}

Now we give explicit expression for $J(-2,p,1)$.
\begin{Lem}\label{lem8}
Let $p \in \mathbb N$, then we have
\begin{align*}
J(-2,p,1)
&=\zeta(p+1)+\zeta(2)+\sum_{j=2}^{p}\zeta(j)\bigg(1+\sum_{i=2}^{p+1-j}(-1)^{1+i}\zeta(i)\bigg)\\
&\quad +\sum_{i=2}^{p}(-1)^{1+i}\bigg((1+\frac{i}{2})\zeta(i+1)
-\frac{1}{2}\sum_{k=1}^{i-2}\zeta(k+1)\zeta(i-k)\bigg)\\
&=2\zeta(2)-\sum_{i=2}^{p}\frac{i}{2}\zeta(i+1)+\frac{1}{2}\sum_{i=3}^{p}\sum_{k=1}^{i-2}\zeta(k+1)\zeta(i-k)\,.
\end{align*}
\end{Lem}
\begin{proof}
From the definition of $J(-2,p,1)$, we can write
\begin{align*}
J(-2,p,1)
&=\sum_{n=1}^\infty \frac{1}{n} \int_{0}^{1}x^{n-2} Li_{p}(x)\mathrm{d}x\\
&=\zeta(p+1)+\sum_{n=1}^\infty \frac{1}{n+1}\bigg(\sum_{j=2}^{p}\frac{(-1)^{p-j}}{n^{p+1-j}}\zeta(j)
+\frac{(-1)^{p-1}}{n^{p}}H_{n}\bigg)\\
&=\zeta(p+1)+\sum_{j=2}^{p}(-1)^{p-j}\zeta(j)\bigg(\sum_{i=2}^{p+1-j}(-1)^{p+1-j-i}\zeta(i)+(-1)^{p-j}\bigg)\\
&\quad +(-1)^{p-1}\sum_{n=1}^\infty H_{n}\bigg(\sum_{i=2}^{p}(-1)^{p-i}\frac{1}{n^{i}}+(-1)^{p-1}\frac{1}{n(n+1)}\bigg)\,.
\end{align*}
It is known that \cite{Sofo}
$$
\sum_{n=1}^\infty \frac{H_{n}}{n(n+1)}=\zeta(2)\,,
$$
thus we get the first result.

On the contrary,
\begin{align*}
J(-2,p,1)
&=-\sum_{n=1}^\infty \frac{1}{n^{p}} \int_{0}^{1}x^{n-2} \log(1-x)\mathrm{d}x\\
&=\zeta(2)+\sum_{n=2}^\infty H_{n-1}\bigg(\sum_{i=2}^{p}\frac{-1}{n^{i}}+\frac{1}{n(n-1)H_{n}}\bigg)\\
&=\zeta(2)+\sum_{n=1}^\infty \frac{H_{n}}{n(n+1)}-\sum_{i=2}^{p}\sum_{n=1}^\infty \frac{H_{n}}{n^{i}}
+\sum_{i=2}^{p}\zeta(i+1)\,,
\end{align*}
thus we get the second result.
\end{proof}

Now we derive explicit expression for $J(m,p,q)$.
\begin{theorem}\label{maintheorem3}
Let $p, q \in \mathbb N$ with $p \geq q$ and $m \in \mathbb N_{0}\cup\{-2\}$, then we have
\begin{align*}
&\quad J(m,p,q)\\
&=\sum_{x=1}^{q-1}\sum_{i_{1}=0}^{p-2}\frac{(-1)^{i_{1}+1}}{(m+1)^{i_{1}+1}}\cdots
\sum_{i_{x-1}=0}^{p-i_{1}-\cdots-i_{x-2}-2}\frac{(-1)^{i_{x-1}+1}}{(m+1)^{i_{x-1}+1}}\\
&\quad \times \sum_{i_{x}=0}^{p-i_{1}-\cdots-i_{x-1}-2}\frac{(-1)^{i_{x}}}{(m+1)^{i_{x}+1}}
\zeta(p-i_{1}-\cdots-i_{x})\zeta(q-x+1)\\
&\quad +\sum_{x=1}^{q-1}\frac{(-1)^{p-2+x}}{(m+1)^{p-2+x}}J(m,1,q-x+1)
\sum_{i_{1}=0}^{p-2}\cdots\sum_{i_{x-1}=0}^{p-i_{1}-\cdots-i_{x-2}-2}1\\
&\quad +\sum_{i_{1}=0}^{p-2}\frac{(-1)^{i_{1}+1}}{(m+1)^{i_{1}+1}}\cdots
\sum_{i_{q-1}=0}^{p-i_{1}-\cdots-i_{q-2}-2}\frac{(-1)^{i_{q-1}+1}}{(m+1)^{i_{q-1}+1}}
J(m,p-i_{1}-\cdots-i_{q-1},1)\,.
\end{align*}
where $K(m+p+x-1,0,q-x+1)$ and $K(m+i_{1}+\cdots+i_{q},p-i_{1}-\cdots-i_{q}+q,0)$ are given in Lemmata \ref{lem6}, \ref{lem7} and \ref{lem8}. Therefore $J(m,p,q)$ can be reduced to zeta values and generalized harmonic numbers.
\end{theorem}
\begin{proof}
It is known that \cite{Freitas}
$$
J(m,p,q)=\frac{\zeta(p)\zeta(q)}{m+1}-\frac{1}{m+1}\bigg(J(m,p-1,q)+J(m,p,q-1)\bigg)\,,
$$
successive application of the above relation $p-1$ times, we can obtain the following recurrence relation:
\begin{align*}
J(m,p,q)
&=\sum_{i_{1}=0}^{p-2}\frac{(-1)^{i_{1}}}{(m+1)^{i_{1}+1}}
\bigg(\zeta(p-i_{1})\zeta(q)-J(m,p-i_{1},q-1)\bigg)\\
&\quad +\frac{(-1)^{p-1}}{(m+1)^{p-1}}J(m,1,q)\,.
\end{align*}
Successive application of the new relation gives
\begin{align*}
&\quad J(m,p,q)\\
&=\sum_{x=1}^{q-1}\sum_{i_{1}=0}^{p-2}\frac{(-1)^{i_{1}+1}}{(m+1)^{i_{1}+1}}\cdots
\sum_{i_{x-1}=0}^{p-i_{1}-\cdots-i_{x-2}-2}\frac{(-1)^{i_{x-1}+1}}{(m+1)^{i_{x-1}+1}}\\
&\quad \times \sum_{i_{x}=0}^{p-i_{1}-\cdots-i_{x-1}-2}\frac{(-1)^{i_{x}}}{(m+1)^{i_{x}+1}}
\zeta(p-i_{1}-\cdots-i_{x})\zeta(q-x+1)\\
&\quad +\sum_{x=1}^{q-1}\frac{(-1)^{p-2+x}}{(m+1)^{p-2+x}}J(m,1,q-x+1)
\sum_{i_{1}=0}^{p-2}\cdots\sum_{i_{x-1}=0}^{p-i_{1}-\cdots-i_{x-2}-2}1\\
&\quad +\sum_{i_{1}=0}^{p-2}\frac{(-1)^{i_{1}+1}}{(m+1)^{i_{1}+1}}\cdots
\sum_{i_{q-1}=0}^{p-i_{1}-\cdots-i_{q-2}-2}\frac{(-1)^{i_{q-1}+1}}{(m+1)^{i_{q-1}+1}}
J(m,p-i_{1}-\cdots-i_{q-1},1)\,.
\end{align*}
\end{proof}

Freitas \cite{Freitas} gave the following recurrence relation for $K(r,0,q)$: For $r\geq 1$, $q\geq 2$, one has
\begin{align*}
K(r,0,q)
=(-1)^{r+q}\frac{r!}{(q-1)!}K(q-1,0,r+1)+(-1)^{r}r!\bigg(\zeta(r+1)\zeta(r+q+1)\bigg)\,.
\end{align*}
From this, Freitas showed that $K(r,p,q)$ could be reduced to zeta values as $p+q+r$ even. We now provide an explicit formula for $K(r,p,q)$.
\begin{theorem}\label{maintheorem4}
Let $p, q, m \in \mathbb N$ with $p \geq q$, then we have
\begin{align*}
&\quad K(m,p,q)\\
&=\sum_{x=1}^{q}K(m+p+x-1,0,q-x+1)\frac{(-1)^{p+x-1}}{(m+1)_{p+x-1}}\sum_{i_{1}=1}^{p}\cdots
\sum_{i_{x-1}=1}^{p-i_{1}-\cdots-i_{x-2}+x-2}1\\
&\quad +\sum_{i_{1}=1}^{p}\frac{(-1)^{i_{1}}}{(m+1)_{i_{1}}}\cdots\sum_{i_{q}=1}^{p-i_{1}-\cdots-i_{q-1}+q-1}
\frac{(-1)^{i_{q}}}{(m+i_{1}+\cdots+i_{q-1}+1)_{i_{q}}}\\
&\quad \times K(m+i_{1}+\cdots+i_{q},p-i_{1}-\cdots-i_{q}+q,0)\,.
\end{align*}
Therefore when $m+p+q$ is even, $K(m,p,q)$ can be reduced to zeta values and generalized harmonic numbers.
\end{theorem}
\begin{proof}
It is known that \cite{Freitas}
$$
K(m,p,q)=-\frac{1}{m+1}\bigg(K(m+1,p-1,q)+K(m+1,p,q-1)\bigg)\,,
$$
successive application of the above relation $p-1$ times, we can obtain the following recurrence relation:
\begin{align*}
K(m,p,q)
=\sum_{i_{1}=1}^{p}\frac{(-1)^{i_{1}}}{(m+1)_{i_{1}}}K(m+i_{1},p-i_{1}+1,q-1)
+\frac{(-1)^{p}}{(m+1)^{p}}K(m+p,0,q)\,.
\end{align*}
Successive application of the new relation gives
\begin{align*}
&\quad K(m,p,q)\\
&=\sum_{x=1}^{q}K(m+p+x-1,0,q-x+1)\frac{(-1)^{p+x-1}}{(m+1)_{p+x-1}}\sum_{i_{1}=1}^{p}\cdots
\sum_{i_{x-1}=1}^{p-i_{1}-\cdots-i_{x-2}+x-2}1\\
&\quad +\sum_{i_{1}=1}^{p}\frac{(-1)^{i_{1}}}{(m+1)_{i_{1}}}\cdots\sum_{i_{q}=1}^{p-i_{1}-\cdots-i_{q-1}+q-1}
\frac{(-1)^{i_{q}}}{(m+i_{1}+\cdots+i_{q-1}+1)_{i_{q}}}\\
&\quad \times K(m+i_{1}+\cdots+i_{q},p-i_{1}-\cdots-i_{q}+q,0)\,.
\end{align*}
Note that
\begin{align*}
K(m,0,q)=m!(-1)^{m}\bigg(S_{q,m+1}^{+,+}-\zeta(m+q+1)\bigg)\,,\quad \cite{Freitas}
\end{align*}
and
\begin{align*}
S_{p,q}^{+,+}
&=\zeta(p+q)\bigg(\frac{1}{2}-\frac{(-1)^{p}}{2}\binom{p+q-1}{p}-\frac{(-1)^{p}}{2}\binom{p+q-1}{q}\bigg)\\
&\quad +\frac{1-(-1)^{p}}{2}\zeta(p)\zeta(q)+(-1)^{p}\sum_{k=1}^{\lfloor \frac{p}{2} \rfloor}\binom{p+q-2k-1}{q-1}\zeta(2k)\zeta(m-2k)\\
&\quad +(-1)^{p}\sum_{k=1}^{\lfloor \frac{q}{2} \rfloor}\binom{p+q-2k-1}{p-1}\zeta(2k)\zeta(m-2k)\,,\quad (p+q \quad \hbox{odd})\cite{Flajolet}
\end{align*}
where $\zeta(1)$ should be interpreted as $0$ whenever it occurs and $\lfloor x \rfloor$ denotes the floor function. Combining the above results, we get the desired result.
\end{proof}

\section{Acknowledgements}

The author is grateful to the referee for her/his useful comments and suggestions. The author is also grateful to Dr. Wanfeng Liang and Dr. Ke Wang for some useful discussions.

\end{document}